\documentclass[12pt]{amsart}
\usepackage{macros}
\usepackage{comment}

\begin{document}
\title
[Extending vector bundles on Curves]
{Extending vector bundles on Curves}

\author{Siddharth Mathur}
\address{Mathematisches Institut\\Heinrich-Heine-Universit\"at\\40204 D\"usseldorf, Germany.}
\urladdr{https://sites.google.com/view/sidmathur/home}

\begin{abstract}

\noindent Given a curve in a (smooth) projective variety $C \subset X$ over an algebraically closed field $k$, we show that a vector bundle $V$ on $C$ can be extended to a ($\mu$-stable) vector bundle on $X$ if $\text{rank}(V) \geq \text{dim}(X)$ and $\text{det}(V)$ extends to $X$. 
\end{abstract}

	\maketitle

\section{Introduction}
Understanding which vector bundles on a subvariety extend to an ambient variety is a well studied problem in algebraic geometry. A famous example is the Grothendieck-Lefschetz theorems which considers the case of a complete intersection in a projective variety. A particularly striking consequence of this work is that complete intersections $X \subset \mathbf{P}^n$ with dimension $\text{dim}(X) \geq 3$ always have a Picard group which is freely generated by $\mathcal{O}_{\mathbf{P}^n}(1)|_X$. However, there are many counterexamples to this statement when the hypothesis on the dimension is dropped: indeed, all elliptic curves can be realized as ample divisors in $\mathbf{P}^2$ but they have a non-finitely generated Picard group. The purpose of this paper is to show that this is the only obstruction when the subvariety is a curve and the rank of the vector bundle is sufficiently large.

\begin{Theorem}[Main Theorem] \label{extendcurves} Let $(X, \mathcal{O}_X(1))$ be a projective scheme over an algebraically closed field $k$, $C \subset X$ a $1$-dimensional closed subscheme and $V$ a vector bundle on $C$ with $\mathrm{rank}(V) \geq \mathrm{dim}(X)$. Then $V$ extends to $X$ if and only if $\text{det}(V)$ extends to $X$. If $X$ is assumed to be integral, smooth and $\text{det}(V)$ extends to $X$, then $V$ may be extended to a $\mu$-stable vector bundle on $X$.
\end{Theorem}

\noindent The proof involves two main ingredients:
\begin{enumerate} 
\item Classical avoidance lemmas, Bertini-type arguments, and
\item the theory of elementary transformations due to Maruyama (see \cite{maruyama1982elementary} for a gentle introduction to the technique).
\end{enumerate}

\begin{Example} The main theorem has some surprising consequences. For instance, if $C \subset X$ is as in the statement of the theorem and $L$ is \emph{any} line bundle on $C$, the vector bundle $V=(L \oplus L^{\vee})^{\oplus n}$ extends to a vector bundle on $X$ for sufficiently large $n$. Similarly, if $V$ is \emph{any} vector bundle on $C$ with sufficiently large rank, then $V \otimes V^{\vee}$ and $V \oplus V^{\vee}$ both extend to vector bundles on $X$. In general, the main theorem shows that (smooth) projective schemes have plenty of ($\mu$-stable) vector bundles. \end{Example}

\textbf{Acknowledgments} I would like to thank Lucas Braune, Yajnaseni Dutta, Jack Hall, Andrew Kresch and David Stapleton for helpful comments. I would also like to thank the referee for many useful comments and suggestions. This research was conducted in the framework of the research training group GRK 2240: Algebro-geometric Methods in Algebra, Arithmetic and Topology, which is funded by the Deutsche Forschungsgemeinschaft.

\section{Some Lemmas}

\begin{center} \emph{Let $k$ denote a fixed algebraically closed field.} \end{center}

The first step of the proof of Theorem \ref{extendcurves} boils down to a simple observation: the dual of any globally generated vector bundle $V$ on a curve $C$ can be realized as an elementary transformation of the trivial bundle along a Cartier divisor $D \subset C$. This is the content of the following:

\begin{lemma} \label{trivial} Let $C$ be a $1$-dimensional scheme which is proper over $k$ and suppose $Z \subset C$ is a finite set containing the associated points of $C$. Moreover, let $V$ denote a rank $r$ vector bundle on $C$ which is globally generated, then there exists an effective Cartier divisor $D$ (which can be viewed as a closed subscheme $i: D \to C$ in a natural way) which doesn't meet $Z$, as well as an exact sequence
\[0 \to V^{\vee} \to \mathcal{O}^{\oplus r}_{C} \to i_*\mathcal{O}_D \to 0.\]
\end{lemma}

\begin{proof} By \cite[Lemma 27]{Experiments} there is a Cartier divisor $i: D \subset C$ not meeting $Z$, a line bundle $L$ on $D$ and an exact sequence
\[0 \to \mathcal{O}^{\oplus r}_{C} \to V \to i_*L \to 0.\]
Since $D$ is a finite scheme, $L$ is trivial and so we have
\[0 \to \mathcal{O}^{\oplus r}_{C} \to V \to i_*\mathcal{O}_D \to 0.\]
Dualizing this sequence we obtain:
\[0 \to V^{\vee} \to \mathcal{O}^{\oplus r}_C \to \underline{\text{Ext}}^1(i_*\mathcal{O}_D, \mathcal{O}_C) \to 0.\]
Indeed, the first term $\underline{\text{Hom}}(i_*\mathcal{O}_D,\mathcal{O}_C)$ vanishes because $D$ contains no associated points of $C$ and the term $\underline{\text{Ext}}^1(V, \mathcal{O}_C)$ vanishes since $V$ is locally free. To determine $\underline{\text{Ext}}^1(i_*\mathcal{O}_D, \mathcal{O}_C)$ one can dualize the exact sequence
\[0 \to \mathcal{O}_C(-D) \to \mathcal{O}_C \to i_*\mathcal{O}_D \to 0\]
to see that 
\[\underline{\text{Ext}}^1(i_*\mathcal{O}_D, \mathcal{O}_C) \simeq \text{coker}(s_D:\mathcal{O}_C \to \mathcal{O}_Y(D)) \simeq i_*\mathcal{O}_D \otimes \mathcal{O}_C(D)\]
where $s_D$ is the section corresponding to $D$. Since $D$ is a finite scheme this is just isomorphic to $i_*\mathcal{O}_D$.
\end{proof}

The following lemma is a refinement of an argument in \cite[Theorem 5.2.5]{huybrechts2010geometry} which, itself, is adapted from the appendix of \cite{maruyama1978moduli}. It will allow us to guarantee our construction produces a $\mu$-stable vector bundle.

\begin{lemma} \label{stable} Fix an integral, smooth, projective $k$-scheme $(X, \mathcal{O}_X(1))$ of dimension $n$ and an effective Cartier divisor $H \subset X$ which is irreducible and generically smooth. Let $E$ be a vector bundle of rank $r \geq 2$ on $X$ which fits into an exact sequence
\[0 \to E \to \mathcal{O}_X^{\oplus r} \to \mathcal{O}_H(N) \to 0\]
for some integer $N>0$. Moreover, suppose the surjection  $\mathcal{O}_X^{\oplus r} \to \mathcal{O}_H(N)$ does not factor through any torsion-free sheaf $F$ with $\mu(F) \leq \frac{r-1}{r}\mathrm{deg}(\mathcal{O}_X(H))$ and $\mathrm{rk}(F)<r$. Then $E$ is a $\mu$-stable sheaf on $X$. \end{lemma}

\begin{proof} Let $0 \neq E' \subset E$ be a saturated proper subsheaf and set $F'$ equal to the saturation of $E'$ in $\mathcal{O}_X^{\oplus r}$. Since $E'$ is saturated in $E$, it follows that the natural map $j: F'/E' \to \mathcal{O}_H(N)$ is injective. First suppose that $F'/E'$ is nonzero. It follows that the injection $j$ induces an isomorphism $(F'/E')_{\eta} \cong \mathcal{O}_H(N)_{\eta}$ at the (unique) generic point $\eta$ of $H$ in $X$. Indeed, $\mathcal{O}_H(N)$ is a line bundle on an irreducible and generically smooth scheme so $\mathcal{O}_{H}(N)_{\eta} \cong \mathcal{O}_{H, \eta}=k(H)$ for a field $k(H)$. Moreover, since $H$ is Cohen-Maucalay, it has no embedded associated points which implies the localization $(F'/E')_{\eta}$ is a nonzero vector space over $k(H)$ (see, for instance, \cite[Tag 0B3L]{stacks}). It follows that the induced map $(F'/E')_{\eta} \cong \mathcal{O}_H(N)_{\eta}$ is an isomorphism. Therefore we have an exact sequence
\[0 \to F'/E' \to \mathcal{O}_H(N) \to B \to 0\]
where the support of $B$ has codimension $\geq 2$ in $X$, so the determinant of $B$ in $X$ is trivial. From this, it follows that $\text{det}(F'/E') \cong \text{det}(\mathcal{O}_H(N)) \cong \mathcal{O}_X(H)$ and $\text{det}(E') \cong \text{det}(F') \otimes \mathcal{O}_X(-H)$. Therefore:
\[\mu(E')=\frac{\text{deg}(\text{det}(F'))+\text{deg}(\mathcal{O}_X(-H))}{\text{rk}(E')}=\mu(F')-\frac{\text{deg}(\mathcal{O}_X(H))}{\text{rk}(E')} < 0 - \frac{\text{deg}(\mathcal{O}_X(H))}{\text{rk}(E)}=\mu(E)\]
where the second equality follows because $\text{rk}(E')=\text{rk}(F')$, the inequality follows because $\mu(F') \leq \mu(\mathcal{O}_X^{\oplus r}) \leq 0$ and $\text{rk}(E') < \text{rk}(E)$, as desired. 

On the other hand, if $F'/E'=0$, then $F=\mathcal{O}_X^{\oplus r}/E'$ is torsion-free and $\phi:\mathcal{O}_X^{\oplus r} \to \mathcal{O}_H(N)$ factors through the quotient $\mathcal{O}_X^{\oplus r} \to F$. By hypothesis, this implies $\mu(F)>\frac{r-1}{r}\text{deg}(\mathcal{O}_X(H))$ and so
\begin{align*}
\mu(E')=\frac{\text{deg}(\text{det}(E'))}{\text{rk}(E')}=\frac{-\text{deg}(\text{det}(F))}{\text{rk}(E')}
& =
\frac{-(r-\text{rk}(E'))}{\text{rk}(E')}\mu(F) \\
& <
\frac{-(r-\text{rk}(E'))}{\text{rk}(E')}\frac{(r-1)}{r}\text{deg}(\mathcal{O}_X(H)) \\
& \leq
\frac{-\text{deg}(\mathcal{O}_X(H))}{r} \\
& =
\mu(E) \\
\end{align*}
The first inequality follows because $\mu(F)>\frac{r-1}{r}\text{deg}(\mathcal{O}_X(H))$ and the second follows because $E'$ is a saturated \emph{proper} subsheaf of $E$, so $\frac{(r-\text{rk}(E'))(r-1)}{\text{rk}(E')} \geq 1$. It follows that $E$ is $\mu$-stable.\end{proof}

To be able to apply the previous lemma, we will need to find a Cartier divisor $H \subset X$ which is irreducible and generically smooth. This can be arranged by a standard Bertini-type result. However, the statement we require could not be found in the literature and so we include a proof for the sake of completeness.

Recall that if $X$ is a proper $k$-scheme, $L$ a line bundle on $X$ and $N \subset \H^0(X, L)$ a linear system, then one says $N$ \emph{separates tangent vectors on} $U$ if the natural map 
\[\text{Ker}[N \to L_p/\text{m}_pL_p] \to \text{m}_pL_p/\text{m}_p^2L_p\]
is surjective for every $p \in U$. As the term suggests, if $N$ separates tangent vectors on $U \subset X$ and has no base points along $U$, then the associated map $f: U \to \mathbf{P}^n$ induces injective differentials
\[\text{d}f_p: T_{U/k,p} \hookrightarrow T_{\mathbf{P}^n/k,f(p)}\]
and so by \cite[Tag 0B2G]{stacks}, the map $f$ is unramified.

\begin{lemma} \label{bertinitangent} Let $(X, \mathcal{O}_X(1))$ be an integral, smooth, and projective $k$-scheme and fix proper closed subschemes $D \subset C \subset X$ with $\text{codim}_X D \geq 2$ and $C$ a $1$-dimensional subscheme. If a linear system $N \subset \H^0(X,L)$ is base point free on $X \backslash D$ and separates tangent vectors on $X \backslash C$, then the vanishing of the general member $s \in N$ is an irreducible Cartier divisor which is smooth away from $C$. \end{lemma}

\begin{proof} Since $N$ is base point free away from $D$ and separates tangent vectors away from $C$, a choice of basis of $N$ defines a morphism $f: X \backslash D \to \mathbf{P}^n$ which is unramified over $X \backslash C$. Since $X \backslash D$ is geometrically integral, it follows that the general hyperplane section in $X \backslash D$ is geometrically irreducible by \cite[Corollaire 6.11.3]{jBertini}. Indeed, the fact that $X \backslash C \to \mathbf{P}^n$ is unramified implies $\text{Im}(f|_{X\backslash C}) \geq 2$ and so we may apply the cited Corollary. Similarly, since $X\backslash C$ is smooth and $N$ induces an unramified morphism on $X\backslash C$, the general hyperplane section in $X \backslash C$ is smooth by \cite[Corollaire 6.11.2]{jBertini}. Since the general hyperplane section over $X \backslash D$ (respectively, $X \backslash C$) corresponds to the intersection of the vanishing of the general member of $N$ in $X$ with $X \backslash D$ (respectively, $X \backslash C$), viewing the vanishing of the general member in $X$ yields the result. The only thing to check is that the vanishing of such a general section remains irreducible, but the only components the vanishing over $X$ could gain are those which lie entirely in $D$. This is not possible because a Cartier divisor cannot have components of codimension $\geq 2$.\end{proof}

Recall that the dual of any globally generated vector bundle on a curve $C$ can be written as an elementary transformation of the trivial vector bundle along a Cartier divisor $D \subset C$. In the proof of Theorem \ref{extendcurves}, we will show that one may extend this elementary transformation to one over all of $X$. As such, we need to show that the Cartier divisor $D \subset C$ must be the intersection of a Cartier divisor, $H$, on $X$, with $C$. This is the primary purpose of the following lemma.

\begin{lemma} \label{E-ample} Let $C \subset X$ be the inclusion of a proper $1$-dimensional closed subscheme in a projective $k$-scheme $X$. Suppose that $E$ is a rank $r \geq 2$ vector bundle on $C$ whose determinant extends to $X$, then there is an ample line bundle $L$ on $X$ with the following properties 
\begin{enumerate}
\item $E \otimes L|_C$ is globally generated, 
\item $\mathrm{det}(E \otimes L|_C)$ is ample,
\item if $\mathrm{det}(E \otimes L|_C) \cong \mathcal{O}_C(D)$ for some effective Cartier divisor $D \subset C$ not containing any associated point of $X$, then there is an ample Cartier divisor $H \subset X$ with $H \cap C=D$ scheme-theoretically, and
\item if $X$ is integral, smooth and projective we may also take $H$ in (3) to be irreducible and generically smooth. 
\end{enumerate}
\end{lemma}

\begin{proof} The existence of $L$ satisfying the conditions (1)-(3) follows from \cite[Lemma 28]{Experiments}. Assume $X$ is integral and smooth, we will show that $L$ can be chosen so that (4) is satisfied. Let $L'$ denote a line bundle on $X$ with $L'|_C \cong \text{det}(E)$. Begin by choosing an ample line bundle $L$ on $X$ so that
\begin{enumerate} 
\item $E \otimes L|_C$ is globally generated, 
\item $L' \otimes L^{\otimes r}$ is very ample,
\item $\text{H}^1(X,L' \otimes L^{\otimes r} \otimes I_C)=0$, and
\item The linear system $M=\H^0(X,L' \otimes L^{\otimes r} \otimes I_C) \subset \H^0(X, L' \otimes L^{\otimes r})$ has no base points away from $C$ and separates tangent vectors on $X \backslash C$.
\end{enumerate}
It is well-known that the first three conditions can be arranged by choosing $L$ to be a high power of an ample line bundle. For the fourth, just note that we can arrange $L$ to be ample enough so that $L' \otimes L \otimes I_C$ is globally generated and $L^{\otimes r-1}$ is very ample (so, in particular, the associated complete linear system separates tangent vectors on $X$). Then the linear system $\H^0(X,L' \otimes L^{\otimes r} \otimes I_C)$  has no base points on $X \backslash C$ and separates tangent vectors on $X \backslash C$. Note that $\text{det}(E \otimes L|_C) \cong L' \otimes L^{\otimes r}|_C$ and therefore the first two items in the statement of the lemma are now satisfied.

Now suppose $\mathrm{det}(E \otimes L|_C) \cong  \mathcal{O}_C(D)$ for some effective Cartier divisor $D \subset C$. Then take cohomology of the exact sequence
\[0 \to L' \otimes L^{\otimes r} \otimes I_C \to L' \otimes L^{\otimes r} \to (L' \otimes L^{\otimes r})|_C \to 0\]
and use the fact that $\H^1(X,L' \otimes L^{\otimes r} \otimes I_C)=0$ to see that the second map induces a surjection on global sections. In particular, a section $s_D \in \H^0(C,L' \otimes L^{\otimes r}|_C)$ defining $D \subset C$ can be lifted to a section $s_{H'} \in \H^0(X,L' \otimes L^{\otimes r})$ whose vanishing, $H' \subset X$ intersects $C$ precisely at $D$. However, it is not clear that $H'$ is irreducible or generically smooth. To remedy this, note that the general member of the linear system
\[N=\langle s_{H'} \rangle +\H^0(X,L' \otimes L^{\otimes r} \otimes I_C) \subset \H^0(X, L' \otimes L^{\otimes r})\]
is smooth away from $C$ and irreducible by lemma \ref{bertinitangent}. Indeed, the linear system $N$ contains $\H^0(X, L' \otimes L^{\otimes r} \otimes I_C)$ and $s_{H'}$ and therefore is base point free away from $D$ and separates tangent vectors away from $C$. Lastly, observe that for the general $s \in N$, $s|_C$ and $s_D$ cut out the same Cartier divisor since it has the form $s=\lambda s_{H'}+s'$ where $s'|_C=0$ and $\lambda \neq 0$.
Thus a general $s \in N$ defines an effective Cartier divisor $H$ on $X$ which is irreducible, generically smooth and has $H \cap C=D$ scheme-theoretically. \end{proof}

We are almost ready to construct an elementary transformation of $X$ along $H$. However, to guarantee $\mu$-stability, we need to show that the hypothesis of lemma \ref{stable} is satisfied. The next lemma shows that there are enough points in the space of maps $\text{Hom}(\mathcal{O}_X^{\oplus r}, \mathcal{O}_H(N))$ which satisfy the desired properties.

Before we state the following lemma, we introduce some basic notation. If $F$ is a coherent sheaf on a scheme $Y$, then define
\[\mathbb{A}(F)=\underline{\text{Spec}}_{\mathcal{O}_Y}(\text{Sym}^*F^{\vee})\]
to be the associated linear scheme. However, if $Y=\Spec A$ and $V=\tilde{M}$ for an $A$-module $M$, then we will write $\mathbb{A}(M)$ instead.

%\begin{lemma} \label{quotients} Fix a positive integer $r$, a projective $k$-scheme $X$, a coherent sheaf $J$, and a finite-type Quot scheme with universal quotient
%\[\mathcal{O}_{X \times Q}^{\oplus r} \to F_{\text{univ}}\].
%Then there exists a monomorphism of linear $Q$-schemes
%\[i: Y \to \mathbb{A}(\underline{\text{Hom}}(\mathcal{O}_{X}^{\oplus r}, J)) \times_k Q \]
%so that for each $q \in Q$, we have natural isomorphisms 
%\[Y_q=\text{Hom}((F_{\text{univ}})_q,J_q)\]
%and if $f: \mathcal{O}_X^{\oplus r} \to J$ is not in the image of $\text{p}_1 \circ i$ then it $f$ does not factor through any sheaf in %$\mathcal{B}$.\end{lemma}

\begin{lemma} \label{lem:avoidingtorsionfrees} Let $(X,\mathcal{O}_X(1))$ denote an integral $k$-projective variety and let $H$ be an ample Cartier divisor with a finite subscheme $D \subset H$. Fix integers $r$ and $\rho$, then there exists an $N_0$ such that for all $N \geq N_0$ and any $\psi \in \mathrm{Hom}(\mathcal{O}_X^{\oplus r},\mathcal{O}_H(N))$ there is a nonempty open subscheme
\[U_N \subset \psi + \mathbb{A}(\mathrm{Hom}(\mathcal{O}_X^{\oplus r}, \mathcal{O}_H(N) \otimes I_D )) \subset \mathbb{A}(\mathrm{Hom}(\mathcal{O}_X^{\oplus r},\mathcal{O}_H(N)))\]
such that all points of $U_N$ correspond to maps $\phi: \mathcal{O}_X^{\oplus r} \to \mathcal{O}_H(N)$ which do not factor through a torsion-free quotient $F$ of $\mathcal{O}_X^{\oplus r}$ with $\mu(F) \leq \rho$ and $\text{rk}(F)<r$.
\end{lemma}

\begin{proof} Consider the family $\mathcal{F}$ of all torsion-free quotients of $\mathcal{O}_X^{\oplus r}$ with $\mu(F) \leq \rho$ and $\text{rk}(F)<r$. By \cite[Lemma 1.7.9]{huybrechts2010geometry}, there is a finite type Quot-scheme $Q$ of $\mathcal{O}_X^{\oplus r}$ so that each $F \in \mathcal{F}$ appears as a fiber of the universal quotient on $X \times Q$:
\[u: \mathcal{O}_{X \times Q}^{\oplus r} \to F_{\text{univ}}.\]
Since a Quot-scheme parametrizes flat families of quotients, the universal object $F_{\text{univ}}$ is $Q$-flat. Now, throw out all the connected components of $Q$ which do not contain a quotient belonging to $\mathcal{F}$. Since $F_{\text{univ}}$ and $\text{p}_1^*\mathcal{O}_H(N)$ are flat over $Q$ for every $N \geq 0$ (here $\text{p}_1: X \times Q \to X$ is the projection), all the fibers $(F_{\text{univ}})_q$ have rank less than $r$ and \cite[Corollaire 7.7.8]{EGAIII.2} implies the functors $G_N,H_N: \underline{\text{Aff}}/Q \to \underline{\text{Set}}$ defined by
\[G_N(T)=\text{Hom}_{\mathcal{O}_{X_T}}(F_{\text{univ}}|_T,\text{p}_1^*\mathcal{O}_H(N)|_T)\]
\[H_N(T)=\text{Hom}_{\mathcal{O}_{X_T}}(\text{p}_1^*\mathcal{O}_{X}^{\oplus r}|_T, \text{p}_1^*\mathcal{O}_H(N)|_T)\]
for any affine $Q$-scheme $T$, are representable by linear schemes. Let $Y_N$ denote the linear $Q$-scheme representing $G_N$ and note that, by cohomology and base change, for all sufficiently large $N$, $H_N$ is naturally represented by the geometric vector bundle $\mathbb{A}(\text{Hom}_{\mathcal{O}_{X}}(\mathcal{O}_{X}^{\oplus r}, \mathcal{O}_H(N))) \times_k Q$. Thus, the quotient $u$ induces a monomorphism
\[g: Y_N \hookrightarrow \mathbb{A}(\text{Hom}_{\mathcal{O}_{X}}(\mathcal{O}_{X}^{\oplus r}, \mathcal{O}_H(N))) \times_k Q\]
which can be interpreted as follows: at a closed point $q \in Q$, the morphism $g$ is the inclusion of homomorphisms $\psi:\mathcal{O}_X^{\oplus r} \to \mathcal{O}_H(N)$ which factor through $(F_{\text{univ}})_q$. Thus, to prove the lemma it suffices to show there is a $N_0$ so that for every $N \geq N_0$ 
\[\text{dim} Y_N<\text{dim}(\mathbb{A}(\text{Hom}(\mathcal{O}_X^{\oplus r},\mathcal{O}_H(N) \otimes I_D))).\]
Indeed, in that case the composition $\pi_1 \circ g$:
\[Y_N \hookrightarrow \mathbb{A}(\text{Hom}(\mathcal{O}_X^{\oplus r}, \mathcal{O}_H(N))) \times_k Q \to  \mathbb{A}(\text{Hom}(\mathcal{O}_X^{\oplus r}, \mathcal{O}_H(N)))\] 
cannot have an image whose closure contains a coset of $\mathbb{A}(\text{Hom}(\mathcal{O}_X^{\oplus r}, \mathcal{O}_H(N) \otimes I_D))$ (here $\pi_1$ denotes the projection onto the first factor). Thus, the complement of the closure of the image of $\pi_1 \circ g$ can be intersected with any such coset to define the nonempty open:
\[U_N \subset \psi + \mathbb{A}(\text{Hom}(\mathcal{O}_X^{\oplus r}, \mathcal{O}_H(N) \otimes I_D)) \subset \mathbb{A}(\text{Hom}(\mathcal{O}_X^{\oplus r},\mathcal{O}_H(N)))\]
so all $\phi \in U_N$ do not factor through any $F \in \mathcal{F}$.

For all $N \gg 0$:
\[\text{dim} Y_N \leq \text{dim}(Q)+\text{max}_{q \in Q}\{\chi(\underline{\text{Hom}}((F_{\text{univ}})_q,(\text{p}_1^*\mathcal{O}_H(N))_q))\}\]
because there is a canonical identification of fibers
\[(Y_N)_q=\mathbb{A}( \text{Hom}((F_{\text{univ}})_q,(\text{p}_1^*\mathcal{O}_H(N))_q))\] 
for any point $q \in Q$.
Moreover, for every $q \in Q$ and all $N \geq 0$, we have:
\[\chi(\mathcal{O}_H^{\oplus r}(N))=\chi(\underline{\text{Hom}}((F_{\text{univ}})_q,(\text{p}_1^*\mathcal{O}_H(N))_q))+p_q(N)\]
for some nonconstant polynomial $p_q(t)$ since $(F_{\text{univ}})_q$ has rank $< r$. Indeed, $p_q(t)$ is the Hilbert polynomial of the cokernel of 
\[0 \to \underline{\text{Hom}}((F_{\text{univ}})_q,(\text{p}_1^*\mathcal{O}_H)_q)) \to \underline{\text{Hom}}((\mathcal{O}_{X \times Q}^{\oplus r})_q, (\text{p}_1^*\mathcal{O}_H)_q)\]
 and because this cokernel is supported on the positive-dimensional projective scheme $H$, $p_q(t) \to \infty$ as $t$ gets large. Next, we claim that the set $\{p_q(t)\}_{q \in Q}$ is finite. Indeed, this is a consequence of the finiteness of the Hilbert polynomials associated to the sheaves $\{\underline{\text{Hom}}((F_{\text{univ}})_q,(\text{p}_1^*\mathcal{O}_H)_q)\}_{q \in Q}$.
This follows by Noetherian induction on $Q$, generic $Q$-flatness of $\underline{\text{Hom}}(F_{\text{univ}},\text{p}_1^*\mathcal{O}_H)$ (see \cite[Tag 052A]{stacks}) and \cite[Lemma 6.8]{ArtinI}. 
Lastly, since $D$ is a finite subscheme of $H$, there is a fixed constant $d>0$ with
\[\chi(\mathcal{O}_H^{\oplus r}(N) \otimes I_D)+d=\chi(\mathcal{O}_H^{\oplus r}(N) )\]
for all large $N$. Thus, by the finiteness of the set $\{p_q(t)\}_{q \in Q}$, there is a $N_0 \gg 0$ so that 
\begin{equation*}
\begin{split}
\text{dim}Y_N & \leq \text{dim}(Q)+\text{max}_{q \in Q}\{\chi(\underline{\text{Hom}}((F_{\text{univ}})_q,(\text{p}_1^*\mathcal{O}_H(N))_q ))\} \\
& \leq 
\text{dim}(Q)+\chi(\mathcal{O}_H(N)^{\oplus r})-\text{min}_{q \in Q}\{p_q(N)\} \\
& =
\text{dim}(Q)+\chi(\mathcal{O}_H(N)^{\oplus r} \otimes I_D)+d-\text{min}_{q \in Q}\{p_q(N)\} \\
& <
\text{dim}(\mathbb{A}(\text{Hom}(\mathcal{O}_X^{\oplus r},\mathcal{O}_H(N) \otimes I_D))) \\
\end{split}
\end{equation*}
for all $N \geq N_0$, as desired. \end{proof}

\section{Proof of the main theorem}

\begin{proof}[\textbf{Proof of Theorem \ref{extendcurves}}] Let $\mathcal{O}_X(1)$ denote an ample line bundle on $X$. By replacing $V$ with $V(m)=V \otimes \mathcal{O}_X(1)^{\otimes m}|_C$ for $m\gg 0$ we may suppose that the conclusion of lemma \ref{E-ample} holds. In particular, $V$ is globally generated. Thus, by lemma \ref{trivial} there exists a Cartier divisor $i: D \to C$ which misses the associated points of $X$ and the associated points of $C$ with the property that there is an exact sequence
\[0 \to V^{\vee} \to \mathcal{O}_C^{\oplus r} \to i_*\mathcal{O}_D \to 0.\]
By adjunction, the surjection on the right, call it $\phi'_D$, is determined by the induced map of sheaves on $D$:
\[\phi_D: \mathcal{O}_C^{\oplus r}|_D=\mathcal{O}^{\oplus r}_D \to \mathcal{O}_D.\]
To prove the theorem it suffices to show $V^{\vee}$ extends to $X$. Let $L$ denote a fixed line bundle on $X$ with $L|_C \cong \text{det}(V)$. Observe that $L|_C \cong \mathcal{O}_C(D) \cong \text{det}(\mathcal{O}_D) $ and that $D \subset C$ is a Cartier divisor in $C$ missing the associated points of $X$. Thus, we may apply the full conclusion of lemma \ref{E-ample}. In particular, there is an effective ample Cartier divisor $H \subset X$ with $H \cap C=D$ scheme-theoretically. Moreover, if $X$ is integral and smooth, we may take $H$ to be smooth away from $C$ and irreducible.

The idea will be to extend the elementary transformation of $\mathcal{O}_C^{\oplus r}$ on $C$ along $D$ to an elementary transformation of $\mathcal{O}_X^{\oplus r}$ on $X$ along $H$. To make this precise, first fix an isomorphism $g_1: \mathcal{O}_X(1)|_D \to \mathcal{O}_D$ and note that this induces isomorphisms $g_N: \mathcal{O}_X(N)|_D \cong \mathcal{O}_D$ for every $N>0$. Our goal is to find a surjective morphism $\psi: \mathcal{O}^{\oplus r}_H \to \mathcal{O}_X(N)|_H$ for some $N>0$ so that the following diagram commutes:

\begin{center} \begin{tikzcd} 
\mathcal{O}^{\oplus r}_H|_D \arrow{d}{\text{id}} \arrow{r}{\psi|_D} & \mathcal{O}_X(N)|_D  \arrow{d}{g_N} \arrow{r} & 0 \\
  \mathcal{O}^{\oplus r}_D \arrow{r}{\phi_D} & \mathcal{O}_D \arrow{r} & 0. \\
\end{tikzcd} \end{center}
Once we have found such a $\psi$, compose it with the natural adjunction morphism to obtain $\psi':\mathcal{O}_X^{\oplus r} \to \mathcal{O}_H^{\oplus r} \to \mathcal{O}_X(N)|_H$. Now consider the associated elementary transformation on $X$:
\[0 \to W \to \mathcal{O}_X^{\oplus r} \to \mathcal{O}_X(N)|_H \to 0\]
and observe that because $H$ is a Cartier divisor, $W$ must be locally free. Moreover, note that $\underline{\text{Tor}}^{\mathcal{O}_X}_1(\mathcal{O}_X(N)|_H, \mathcal{O}_C)=0$ since it injects into the vector bundle $W|_C$ and is supported on $D$ (which contains no associated points of $C$). Thus, upon restriction to $C$, the isomorphism $g_N: \mathcal{O}_X(N)|_D \to \mathcal{O}_D$ induces a morphism of short exact sequences:

\begin{center} \begin{tikzcd} 
0 \arrow{r} & W|_C \arrow{r} \arrow{d}{\cong} & \mathcal{O}_X^{\oplus r}|_C \arrow{d}{\text{id}} \arrow{r}{\psi'|_C} & \mathcal{O}_X(N)|_D \arrow{d}{g_N} \arrow{r} & 0 \\
 0 \arrow{r} & V^{\vee} \arrow{r} & \mathcal{O}_C^{\oplus r} \arrow{r}{\phi'_D} & \mathcal{O}_D \arrow{r} & 0 \\
\end{tikzcd} \end{center}

\noindent thereby producing an isomorphism $W|_C \simeq V^{\vee}$, as desired. Note that the isomorphisms $g_N: \mathcal{O}_X(N)|_D \simeq \mathcal{O}_D$ we have fixed determine isomorphisms 
\[\text{\underline{Hom}}(\mathcal{O}_H^{\oplus r}, \mathcal{O}_X(N)|_H) \otimes \mathcal{O}_D \simeq \text{\underline{Hom}}(\mathcal{O}_D^{\oplus r}, \mathcal{O}_D).\]
As such, the rest of the proof is devoted to finding a $\psi$ which restricts to $\phi_D$ via the isomorphism $g_N$ (for some large $N>0$).

Next, take $N$ to be large enough so that the short exact sequence on $H$
\[0 \to \text{\underline{Hom}}(\mathcal{O}_H^{\oplus r}, \mathcal{O}_X(N)|_H) \otimes I_D \to \text{\underline{Hom}}(\mathcal{O}_H^{\oplus r}, \mathcal{O}_X(N)|_H) \to  \text{\underline{Hom}}(\mathcal{O}_D^{\oplus r}, \mathcal{O}_D) \to 0\]
remains exact after taking global sections so that we may lift $\phi_D \in \H^0(H, \text{\underline{Hom}}(\mathcal{O}_D^{\oplus r}, \mathcal{O}_D))$ to a section $\psi_D \in \H^0(H, \text{\underline{Hom}}(\mathcal{O}_H^{\oplus r}, \mathcal{O}_X(N)|_H))$.
The issue is that $\psi_D: \mathcal{O}_H^{\oplus r} \to \mathcal{O}_X(N)|_H$ may not be surjective away from $D$. We will rectify this by adding a factor from $\H^0(H, \text{\underline{Hom}}(\mathcal{O}_H^{\oplus r}, \mathcal{O}_X(N)|_H) \otimes I_D)$ which doesn't change the behavior of $\psi_D$ along $D$. 

For all sufficiently large $N$, we may fix a basis 
\[\psi_1,...,\psi_n \in \H^0(H, \text{\underline{Hom}}(\mathcal{O}_H^{\oplus r}, \mathcal{O}_X(N)|_H) \otimes I_D)\]
so that at any point $p \in H \backslash D$, there is a collection of $r$ sections among the $\psi_1,...,\psi_n$ which form a basis for the vector space $\text{\underline{Hom}}(\mathcal{O}_H^{\oplus r}, \mathcal{O}_X(N)|_H) \otimes k(p)$. Viewing the sections $\psi_D,\psi_1,...,\psi_n$ in $\H^0(H, \text{\underline{Hom}}(\mathcal{O}_H^{\oplus r}, \mathcal{O}_X(N)|_H))$ we set $\mathbf{A}^n_k=\Spec k[x_1,...,x_n]$ and consider the universal section 
\[\psi_{\text{univ}}=\psi_D+\Sigma_{i=1}^n x_i\psi_i\]
of $\psi_D+\text{\underline{Hom}}(\mathcal{O}_H^{\oplus r}, \mathcal{O}_X(N)|_H) \otimes I_D$ pulled back to $\mathbf{A}^n_k \times_k H$. Thus, by construction, the universal section restricts to the section $\psi_{\underline{a}}=\psi_D+\Sigma_{i=1}^na_i\psi_i$ over $\underline{a}=(a_1,...,a_n) \in \mathbf{A}^n_k(k)$. 

Over the complement $U=H \backslash D \subset H$, consider the closed locus of non-surjective maps
\[Z=\{(\underline{a},u)| \psi_{\underline{a}} \otimes k(u) \text{ is not surjective}\} \subset \mathbf{A}^n_k \times U.\]
For any $u \in U$ the fiber $Z_u$ has codimension $r$ since the $\psi_1,...,\psi_n$ generate $\text{\underline{Hom}}(\mathcal{O}_H^{\oplus r}, \mathcal{O}_X(N)|_H)$ at all $u \in U$. Indeed, there is a surjective map
\[\pi: \mathbf{A}^n_{k(u)} \to \Hom_{k(u)}(k(u)^{r},k(u))=\text{\underline{Hom}}(\mathcal{O}_H^{\oplus r}, \mathcal{O}_X(N)|_H) \otimes k(u)\]
sending
\[\underline{a}=(a_1,...,a_n) \mapsto (\psi_D)_{k(u)}+\Sigma_{i=1}^na_i(\psi_i)_{k(u)}\]
and only the zero map doesn't have full rank. Therefore the fiber $\pi^{-1}(0)=Z_u$ has dimension $n-r$ so the dimension of $Z$ (and the closure of its image $\overline{\text{p}_1(Z)}$ in $\mathbf{A}^n_k$) is at most 
\[n-r+\dim(H)<n\]
because $r=\text{rank}(V) \geq \text{dim}(X)>\text{dim}(H)$. 

Thus, there is a point $\underline{c}=(c_1,...,c_n) \in \mathbf{A}^n_k(k)$ avoiding $\overline{\text{p}_1(Z)}$, and we claim that the corresponding section of $\text{\underline{Hom}}(\mathcal{O}_H^{\oplus r},\mathcal{O}_X(N)|_H)$ works as desired. Indeed, a point avoiding $\overline{\text{p}_1(Z)}$ corresponds to a section
\[\psi_{\underline{c}}=\psi_D+\Sigma_{i=1}^n c_i\psi_i \in \H^0(H,\text{\underline{Hom}}(\mathcal{O}_H^{\oplus r},\mathcal{O}_X(N)|_H))\]
which is a surjective linear map for every $u \in U$ (since $(\underline{c},u)$ is not in $Z$). Moreover, on $D$, we have $\psi_{\underline{c}}|_D=\psi_D|_D=\phi_D$. Also, $\phi_D$ is surjective so Nakayama's lemma implies $\psi_{\underline{c}}$ is surjective over all of $H$. Thus, the kernel of $\psi_{\underline{c}}': \mathcal{O}_X^{\oplus r} \to \mathcal{O}_H^{\oplus r} \to \mathcal{O}_H(N)$ is a vector bundle $W$ on $X$ extending $V^{\vee}$. 

If $X$ is smooth, lemma \ref{lem:avoidingtorsionfrees} says that after perhaps enlarging $N$, there is a nonempty open subset 
\[U_N \subset \psi_D+\text{Hom}(\mathcal{O}_H^{\oplus r}, \mathcal{O}_X(N)|_H) \otimes I_D\]
so that the corresponding composition $\mathcal{O}_X^{\oplus r} \to \mathcal{O}_H^{\oplus r} \to \mathcal{O}_X(N)|_H$ does not factor through a torsion-free sheaf $F$ on $X$ with $\text{rk}(F)<r$ and $\mu(F) \leq \frac{r}{r-1}\text{deg}(\mathcal{O}_X(H))$. Therefore, the general section
\[\underline{c} \in \mathbf{A}^n_k \cong \psi_D+\text{Hom}(\mathcal{O}_H^{\oplus r}, \mathcal{O}_X(N)|_H \otimes I_D)\]
induces a surjective map $\psi_{\underline{c}}':\mathcal{O}_X^{\oplus r} \to \mathcal{O}_H(N)$ (because it misses $\overline{\text{p}_1(Z)}$) and the resulting kernel satisfies the hypothesis of lemma \ref{stable} (because $\underline{c} \in U_N$). It follows that the kernel of $\psi_{\underline{c}}'$ is a $\mu$-stable vector bundle on $X$ extending $V^{\vee}$. \end{proof}

%\begin{Example} One can combine this result with Chow's lemma to show that some non-projective schemes (or even algebraic spaces) have plenty of vector bundles as well. Indeed, let $X$ denote an algebraic space which is proper over an algebraically closed field $k$ and suppose that there is a proper birational morphism $f: Y \to X$ such that
%\begin{enumerate}
%\item $Y$ is projective over $k$,
%\item $f$ is an isomorphism away from finitely many points of $X$, and
%\item the dimension of the fibers of $f$ is no larger than $1$.
%then $X$ admits many nontrivial vector bundles. We give a brief sketch of an argument, note that there is a $1$-dimensional subscheme $R \subset Y$ containing the exceptional locus of $f$ such that if $W$ is a vector bundle on $Y$ which is trivial on $R$, then $f_*W$ is locally free on $X$ (see \cite{PerlingToric}[Cor. 2.3]). Now, fix a $1$-dimensional subscheme $C \subset Y$ which strictly contains $R$ and observe that there are plenty of non-trivial vector bundles $V$ on $C$ with trivial determinant which are trivial on $R$. By the main theorem, these extend to vector bundles on $Y$ which must descend to vector bundles on $Y$. \end{Example}

\bibliography{mybib}{}

\begin{thebibliography}{1}

\bibitem{ArtinI}
Michael Artin.
\newblock Algebraization of formal moduli. i.
\newblock {\em Global analysis (papers in honor of K. Kodaira)}, pages 21--71,
  1969.

\bibitem{EGAIII.2}
A.~Grothendieck.
\newblock \'{E}l\'{e}ments de g\'{e}om\'{e}trie alg\'{e}brique. {III}.
  \'{E}tude cohomologique des faisceaux coh\'{e}rents, part 2.
\newblock {\em Inst. Hautes \'{E}tudes Sci. Publ. Math.}, (17):91, 1963.

\bibitem{huybrechts2010geometry}
Daniel Huybrechts and Manfred Lehn.
\newblock {\em The geometry of moduli spaces of sheaves}.
\newblock Cambridge University Press, 2010.

\bibitem{jBertini}
J.P. Jouanolou.
\newblock {\em Th{\'e}or{\`e}mes de Bertini et applications}.
\newblock Progress in mathematics (Birkh{\"a}user) 42. Birkh{\"a}user, 1983.

\bibitem{maruyama1982elementary}
Masaki Maruyama.
\newblock Elementary transformations in the theory of algebraic vector bundles.
\newblock In {\em Algebraic geometry ({L}a {R}\'{a}bida, 1981)}, volume 961 of
  {\em Lecture Notes in Math.}, pages 241--266. Springer, Berlin, 1982.

\bibitem{maruyama1978moduli}
Masaki Maruyama et~al.
\newblock Moduli of stable sheaves, ii.
\newblock {\em Journal of Mathematics of Kyoto University}, 18(3):557--614,
  1978.

\bibitem{Experiments}
Siddharth {Mathur}.
\newblock {Experiments on the Brauer map in High Codimension}.
\newblock page arXiv:2002.12205, February 2020.

\bibitem{stacks}
The {Stacks Project Authors}.
\newblock \emph{{S}tacks {P}roject}.
\newblock http://stacks.math.columbia.edu, 2020.

\end{thebibliography}
\bibliographystyle{plain}

\end{document}